\documentclass[12pt]{article}
\usepackage[english]{babel}
\usepackage{latexsym,amssymb,amsmath,amsthm,amsfonts}
\usepackage{graphicx}

\def\'#1{\ifx#1i{\accent"13 \i}\else{\accent"13 #1}\fi}

\newtheorem{theorem}{Theorem}[section]
\newtheorem{corollary}[theorem]{Corollary}
\newtheorem{lemma}[theorem]{Lemma}

\title{The 4-girth-thickness of the complete multipartite graph}
\author{Christian Rubio-Montiel}
\date{}
\begin{document}
\maketitle

\begin{abstract}
The $g$-girth-thickness $\theta(g,G)$ of a graph $G$ is the smallest number of planar subgraphs of girth at least $g$ whose union is $G$. In this paper, we calculate the $4$-girth-thickness $\theta(4,G)$ of the complete $m$-partite graph $G$ when each part has an even number of vertices.
\end{abstract}


	
\section{Introduction}
The \emph{thickness} $\theta(G)$ of a graph $G$ is the smallest number of planar subgraphs whose union is $G$. Equivalently, it is the smallest number of parts used in any edge partition of $E(G)$ such that each set of edges in the same part induces a planar subgraph.

This parameter was introduced by Tutte \cite{MR0157372} in the 60s. The problem to calculate the thickness of a graph $G$ is an NP-hard problem \cite{MR684270} and a few of exact results can be found in the literature, for example, if $G$ is a complete graph \cite{MR0460162,MR0164339,MR0186573}, a hypercube \cite{MR0211901}, or a complete multipartite graph for some particular values \cite{MR3243852,MR3610769}. Even for the complete bipartite graph there are only partial results \cite{MR0158388,MR0229545}.

Some generalizations of the thickness for complete graphs have been studied, for instance, the outerthickness $\theta_o$, defined similarly but with outerplanar instead of planar \cite{MR1100049}, the $S$-thickness $\theta_S$, considering the thickness on a surface $S$ instead of the plane \cite{MR0245475}, and the $k$-degree-thickness $\theta_k$ taking a restriction on the planar subgraphs: each planar subgraph has maximum degree at most $k$ \cite{MR0491284}.

The thickness has applications in the design of circuits \cite{MR1079374}, in the Ringel's earth-moon problem \cite{MR1735339}, and to bound the achromatic numbers of planar graphs \cite{araujo2017complete}, etc. See the survey \cite{MR1617664}.

In \cite{rubio20174}, the author introduced the \emph{$g$-girth-thickness} $\theta(g,G)$ of a graph $G$ as the minimum number of planar subgraphs of girth at least $g$ whose union is $G$, a generalization of the thickness owing to the fact that the $g$-girth-thickness is the usual thickness when $g=3$ and also the \emph{arboricity number} when $g=\infty$ because the \emph{girth} of a graph is the size of its shortest cycle or $\infty$ if it is acyclic. See also \cite{casta2017}.

In this paper, we obtain the $4$-girth-thickness $\theta(4,K_{n_1,n_2,\dots,n_m})$ of the complete $m$-partite graph $K_{n_1,n_2,\dots,n_m}$ when $n_i$ is even for all $i\in\{1,2,\dots,m\}$.


\section{Calculating $\theta(4,K_{n_1,n_2,\dots,n_m})$}\label{Section2}

Given a simple graph $G$, we define a new graph $G\bowtie G$ in the following way: If $G$ has vertex set $V=\{w_1,w_2,\dots,w_n\}$, the graph $G\bowtie G$ has as vertex set two copies of $V$, namely, $\{u_1,u_2,\dots,u_n,v_1,v_2,\dots,v_n\}$ and two vertices $x_iy_j$ are adjacent if $w_iw_j$ is an edge of $G$, for the symbols $x,y\in\{u,v\}$. For instance, if $w_1w_2$ is an edge of a graph $G$, the graph $G\bowtie G$ has the edges $u_1u_2$, $v_1v_2$, $u_1v_2$ and $v_1u_2$. See Figure \ref{Fig1}.
\begin{figure}[htbp]
\begin{center}	
\includegraphics{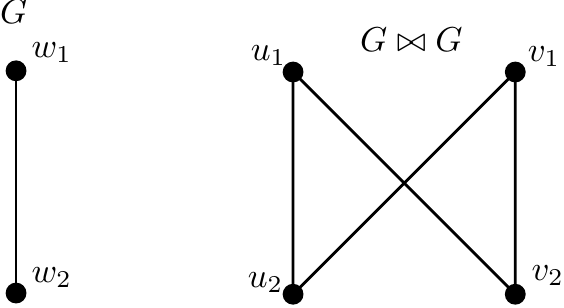}
\caption{\label{Fig1} An edge of $G$ produces four edges in $G\bowtie G$.}
\end{center}
\end{figure}

On the other hand, an acyclic graph of $n$ vertices has at most $n-1$ edges and a planar graph of $n$ vertices and girth $g<\infty$ has at most $\frac{g}{g-2}(n-2)$ edges, see \cite{MR2368647}. Therefore, a planar graph of $n$ vertices and girth at least $4$ has at most $2(n-2)$ edges for $n\geq 4$ and at most $n-1$, otherwise. In consequence, the $4$-girth-thickness $\theta(4,G)$ of a graph $G$ is at least $\left\lceil \frac{|E(G)|}{2(n-2)}\right\rceil$ for $n\geq 4$ and at least $\left\lceil \frac{|E(G)|}{n-1}\right\rceil$, otherwise.
\begin{lemma}\label{lemma1}
If $G$ is a tree of order $n$ then $G\bowtie G$ is a bipartite planar graph of size $2(2n-2)$.
\end{lemma}
\begin{proof}
By induction over $n$. The basis is given in Figure \ref{Fig1} for $n=2$. Now, take a tree $G$ with $n+1$ vertices. Since it has at least a leaf, we say, the vertex $w_1$ incident to $w_2$ then we delete $w_1$ from $G$ and by induction hypothesis, $H\bowtie H$ is a bipartite planar of size $2(2n-2)$ edges for $H=G\setminus \{w_1\}$. Since $H$ is connected, the vertex labeled $w_2$ has at least a neighbour, we say, the vertex labeled $w_3$, then $u_2v_3v_2$ is a path in $H\bowtie H$ and the edge $u_2v_2\notin E(H\bowtie H)$. Add the paths $u_2v_1v_2$ and $u_2u_1v_2$ to $H\bowtie H$ such that both of them are ``parallel'' to $u_2v_3v_2$ and identify the vertices $u_2$ as a single vertex as well as the vertices $v_2$. This proves that $G\bowtie G$ is planar. To verify that is bipartite, given a proper coloring of $H\bowtie H$ with two colors, we extend the coloring putting the same color of $v_3$ to $v_1$ and $u_1$. Then the resulting coloring is proper. Due to the fact that we add four edges, $H\bowtie H$ has $2(2n-2)+4=2(2(n+1)-2)$ edges and the lemma follows.
\end{proof}
Now, we recall that the arboricity number or $\infty$-girth-thickness $\theta(\infty,G)$ of a graph $G$ equals (see \cite{MR0161333})\[\max\left\{ \left\lceil \frac{|E(H)|}{|V(H)|-1}\right\rceil :H\textrm{ is an induced subgraph of }G\right\}.\]
We have the following theorem.
\begin{theorem} \label{teo1}
If $G$ is a simple graph of $n\geq 2$ vertices and $e$ edges, then \[\left\lceil \frac{e}{n-1}\right\rceil \leq \theta(4,G\bowtie G)\leq \theta(\infty,G).\]
\end{theorem}
\begin{proof}
Since $G\bowtie G$ has $2n\geq 4$ vertices, $4e$ edges and
\[\frac{|E(G\bowtie G)|}{2(|V(G\bowtie G)|-2)}=\frac{4e}{2(2n-2)}=\frac{e}{n-1}\]
it follows the lower bound \[\left\lceil \frac{e}{n-1}\right\rceil \leq \theta(4,G\bowtie G).\]
To verify the upper bound, take an acyclic edge partition $\{F_1,F_2,\dots,F_{\theta(\infty,G)}\}$ of $E(G)$. Therefore, $\{F_1	\bowtie F_1,F_2\bowtie F_2,\dots,F_{\theta(\infty,G)}\bowtie F_{\theta(\infty,G)}\}$ is an edge partition of $E(G\bowtie G)$ (where $F_i\bowtie F_i:=E(\left\langle F_{i}\right\rangle \bowtie\left\langle F_{i}\right\rangle )$ and $\left\langle F_{i}\right\rangle $ is the induced subgraph of the edge set $F_i$ for all $i\in\{1,2,\dots,\theta(\infty,G)\}$). Indeed, an edge $x_jy_{j'}\in E(G\bowtie G)$ is in $F_i\bowtie F_i$ if and only if $w_jw_j'\in E(G)$ is in $F_i$. By Lemma \ref{lemma1}, the result follows.
\end{proof}
\begin{corollary} \label{cor1}
If $G$ is a simple graph of $n\geq 2$ vertices and $e$ edges with $\theta(\infty,G)=\left\lceil \frac{e}{n-1}\right\rceil$, then \[\theta(4,G\bowtie G)=\left\lceil \frac{e}{n-1}\right\rceil.\]
\end{corollary}
Next, we estimate the arboricity number of the complete $m$-partite graph.
\begin{lemma} \label{lemma}
If $K_{n_1,n_2,\dots,n_m}$ is the complete $m$-partite graph then $\theta(\infty,G)=\left\lceil \frac{e}{n-1}\right\rceil$ where $n=n_1+n_2+\ldots+n_m$ and $e=n_1n_2+n_1n_3+\ldots+n_{m-1}n_m$.
\end{lemma}
\begin{proof}
By induction over $n$. The basis is trivial for $K_{1,1}$. Let $G=K_{n_1,n_2,\dots,n_m}$ with $n>2$ and $H=G\setminus \{u\}$ a proper induced subgraph of $G$ for any vertex $u$. By the induction hypothesis, $\theta(\infty,H)=\max\left\{ \left\lceil \frac{|E(F)|}{|V(F)|-1}\right\rceil :F\leq H\right\}=\left\lceil \frac{|E(H)|}{(n-1)-1}\right\rceil$, where $F\leq H$ indicates that $F$ is an induced subgraph of $H$. Since $u$ is an arbitrary vertex and by the hereditary property of the induced subgraphs, we only need to show that \[\frac{|E(H)|}{n-2}\leq \frac{e}{n-1}\]
because \[\max\left\{ \left\lceil \frac{|E(F)|}{|V(F)|-1}\right\rceil :F\leq G\right\}=\max\left\{ \left\lceil \frac{e}{n-1}\right\rceil,\left\lceil \frac{|E(H)|}{n-2}\right\rceil :H=G\setminus \{u\},u\in V(G)\right\}.\]
We prove it in the following way. Without loss of generality, $u$ is a vertex in a part of size $n_m$.

Since \[\begin{array}{cccc}
n_{1}+ & n_{1}n_{2}+ & \ldots & +n_{1}n_{m}+\\
 & n_{2}+ & \ldots & +n_{2}n_{m}+\\
 &  & \vdots\\
 &  &  & n_{m-1}n_{m}
\end{array}\leq\begin{array}{cccc}
n_{1}^{2}+ & n_{1}n_{2}+ & \ldots & +n_{1}n_{m}+\\
n_{2}n_{1}+ & n_{2}^{2}+ & \ldots & +n_{2}n_{m}+\\
 &  & \vdots\\
n_{m-1}n_{1}+ & n_{m-1}n_{2}+ & \ldots & +n_{m-1}n_{m}
\end{array}
\]
then $e+n_1+n_2+\ldots+n_{m-1}\leq n(n_1+n_2+\ldots+n_{m-1})$ and
\[en-e-n(n_1+n_2+\ldots+n_{m-1})+(n_1+n_2+\ldots+n_{m-1})\leq en-2e\]
\[(n-1)(e-(n_1+n_2+\ldots+n_{m-1}))\leq e(n-2)\]
\[\frac{|E(H)|}{n-2}\leq\frac{e}{n-1}\]
and the result follows.
\end{proof}
Now, we can prove our main theorem.
\begin{theorem} \label{teo2}
If $G=K_{2n_1,2n_2,\dots,2n_m}$ is the complete $m$-partite graph then $\theta(4,G)=\left\lceil \frac{e}{n-1}\right\rceil$ where $n=n_1+n_2+\ldots+n_m$ and $e=n_1n_2+n_1n_3+\ldots+n_{m-1}n_m$.
\end{theorem}
\begin{proof}
We need to show that $G=K_{n_1,n_2,\dots,n_m}\bowtie K_{n_1,n_2,\dots,n_m}$. Let $(W_1,W_2,\dots,W_m)$ be an $m$-partition of $K_{n_1,n_2,\dots,n_m}$. The graph $K_{n_1,n_2,\dots,n_m}\bowtie K_{n_1,n_2,\dots,n_m}$ has the partition $(U_1\cup V_1,U_2\cup V_2,\dots,U_m\cup V_m)$ where $U_i$ and $V_i$ are copies of $W_i$ for $i\in\{1,2,\dots,m\}$. Take two vertices $x_i$ and $y_j$ in different parts, without loss of generality, $U_1\cup V_1$ and $U_2\cup V_2$. If the vertex $x_i$ is in $U_1$ and $y_j$ is in $U_2$ then they are adjacent because $w_iw_j$is an edge of $K_{n_1,n_2,\dots,n_m}$ is $m$-complete. Similarly for $x_i\in V_1$ and $y_j\in V_2$. If $x_i$ is in $U_1$ and $y_j$ is in $V_2$, then also they are adjacent because $w_iw_j$ is an edge of $K_{n_1,n_2,\dots,n_m}$.
By Corollary \ref{cor1} and Lemma \ref{lemma}, the theorem follows.
\end{proof}
Due to the fact that $\theta(4,G)=\theta(3,G)=\theta(G)$ for any triangle-free graph $G$, we obtain an alternative proof for the thickness of the complete bipartite graph $K_{2n_1,2n_2}$ that is given in \cite{MR0158388}.
\begin{corollary} \label{cor2}
If $G=K_{2n_1,2n_2}$ is the complete bipartite graph then $\theta(G)=\left\lceil \frac{e}{n-1}\right\rceil$ where $n=n_1+n_2$ and $e=n_1n_2$.
\end{corollary}
\section*{Acknowledgments}
The authors wish to thank the anonymous referees of this paper for their suggestions and remarks.

C. Rubio-Montiel was partially supported by PAIDI/007/19.


Divisi{\' o}n de Matem{\' a}ticas e Ingenier{\' i}a\\
FES Acatl{\' a}n\\
Universidad Nacional Aut{\' o}noma de M{\' e}xico\\
Naucalpan 53150\\
State of Mexico\\
Mexico\\
{\tt christian.rubio@apolo.acatlan.unam.mx}

\end{document}